\documentclass[fleqn]{article}

\usepackage{graphics}
\usepackage{epsfig}
\usepackage{times}
\usepackage{amsmath}
\usepackage{amsthm}
\usepackage{amssymb}
\usepackage{amsfonts}

\begin{document}

\title{Cyclic Network Automata and Cohomological Waves}
\author{Yiqing Cai, Robert Ghrist
 }

\maketitle

\newtheorem{theorem}{Theorem}
\newtheorem{definition}{Definition}
\newtheorem{lemma}{Lemma}
\newtheorem{statement}{Statement}
\newtheorem{corollary}{Corollary}
\newtheorem{proposition}{Proposition}

\newcommand{\style}[1]{{\bf{#1}}} 

\newcommand{\ie}{{\em i.e.}}
\newcommand{\eg}{{\em e.g.}}
\newcommand{\cf}{{\em cf.}}

\newcommand{\real}{{\mathbb R}}
\newcommand{\comp}{{\mathbb C}}
\newcommand{\nats}{{\mathbb N}}
\newcommand{\rats}{{\mathbb Q}}
\newcommand{\zed}{{\mathbb Z}}
\newcommand{\euc}{{\mathbb E}}
\newcommand{\Domain}{{\mathcal D}} 
\newcommand{\Alphabet}{{\mathcal A}} 
\newcommand{\Graph}{{\mathcal G}} 
\newcommand{\Tree}{{\mathcal T}} 
\newcommand{\Rips}{{\mathcal R}} 
\newcommand{\Neighbor}{{\mathcal N}} 
\newcommand{\Update}{{\mathcal G}} 
\newcommand{\Statespace}{{\mathbb{Z}_n^X}} 
\newcommand{\cont}{{Cont}} 
\newcommand{\degree}{{deg}} 
\newcommand{\pofs}{{p_s}} 

\begin{abstract}
This paper considers a dynamic coverage problem for sensor networks that are sufficiently dense but not localized. Only a small fraction of sensors may be in an \em{awake} state at any given time. The goal is to find a decentralized protocol for establishing dynamic, sweeping barriers of awake-state sensors. Following Baryshnikov-Coffman-Kwak \cite{gha}, we use network cyclic cellular automata to generate waves. This paper gives a rigorous analysis of network-based cyclic cellular automata in the context of a system of narrow hallways and shows that waves of awake-state nodes turn corners and automatically solve {\em pusuit/evasion}-type problems without centralized coordination.

As a corollary of this work, we unearth some interesting topological interpretations of features previously observed in cyclic cellular automata (CCA). By considering CCA over networks and completing to simplicial complexes, we induce dynamics on the higher-dimensional complex. In this setting, waves are seen to be generated by topological defects with a nontrivial degree (or winding number). The simplicial complex has the topological type of the underlying map of the workspace (a subset of the plane), and the resulting waves can be classified cohomologically. This allows one to ``program'' pulses in the sensor network according to cohomology class. We give a realization theorem for such pulse waves.
\end{abstract}

\section{Introduction}
\label{sec:intro}

A \style{wireless sensor network (WSN)} consists of a collection of sensors networked via wireless communications, with every sensor being a device collecting data of the environment with respect to one or more features, and returning with a signal \cite{wsn_yick, Akyildiz02wirelesssensor}. Sensors can read, {\em inter alia}, temperature, pressure, sound, target presence, range, and identification. Current-generation smart sensors, increasingly smaller in size, can perform data processing and computation, albeit with very limited memory and computation capability. However, constrained by the locality of sensors' sensing function, networks of sensors have many more applications in monitoring larger domains than a single sensor. They can communicate with each other in the sense of transmitting and receiving signals, which allows local information to be collected and agregated globally.

A very common application of wireless sensor networks is intrusion-detection: the network monitors an area, reporting the existence of intruders when they are detected by at least one sensor. Video surveillance provides one such example. There is considerable activity in this field, focusing on different features and goals, "optimizing" networks in various senses. One such aspect concerns coverage problems, which consider whether a domain is always fully covered by the union of sensing regions of sensors, static or mobile. Current approaches include methods from graph theory, computational geometry, and algebraic topology \cite{fekete,916633, Huang, coverage_silva}. Other aspects focus on providing a specific degree of coverage, while keeping the connectivity of the network \cite{ccp}. What concerns us most in the present work is the minimization of energy consumption, keeping in mind that sensors are almost always battery-driven. One of the most intuitive ways is constructing a sleep-wake protocol for the network, allowing sensors to alternate between higher and lower energy cost states \cite{gha}.

\begin{figure}
\centering
\epsfig{file=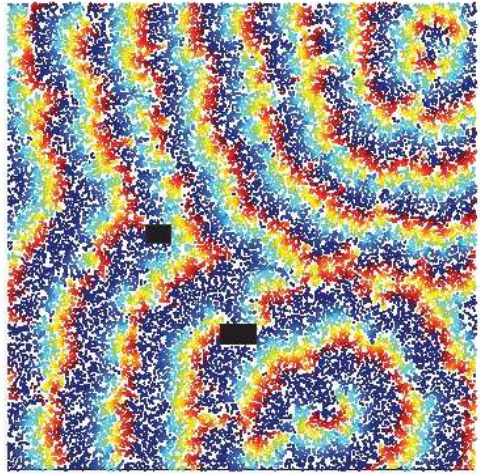,height=3in,width=4in}
\caption{Greenberg-Hastings model in square with blocks: black blocks are regions where no sensors are located, waves behave the same as with no blcks.}
\label{fig:ghm in square}
\end{figure}

This paper is motivated by recent work suggesting the use of cyclic cellular automata (CCA) for intrusion-detection sensor networks \cite{gha}. This inventive paper applied the Greenberg-Hastings automata on a two dimensional plane to generate ``waves'' of on-state sensors for intruder detection. Generally speaking, this specific automata assigns to each sensor the state space $\zed_n$ (the cyclic group on $n$ elements), and the sensors update their state by advancing one automatically, except in state $0$, in which case update to state $1$ is induced by contact with at least one neighbor in state $1$. This Greenberg-Hastings automata on lattices has been frequently investigated in the literature as follows. Some rigorous statistical results has been proved for GHM with state space $\zed_3$ \cite{durrett}. For general state space $\zed_n$, experiments has been carried out in \cite{cca_in2d}, and specific features or patterns that would keep and ergodic behaviors has been studied in \cite{excitableCA,metastability,Durrett93asymptoticbehavior}. A few authors have considered what happens on a graph as opposed to a lattice \cite{Matamala1997161}; this is the starting point for the application in \cite{gha} to sensor networks. The main features and results of Baryshnikov-Coffman-Kwak include:
\begin{enumerate}
\item The CCA runs on a random graph instead of a lattice, where nodes represent sensors and edges communication links.
\item The network is completely non-localized and coordinate-free
\item The CCA with random initial conditions generates the familiar spiral-like \style{wavefronts} that sweep the whole domain with on-state sensors, giving a decentralized scheme for low-power dynamic barrier coverage.
\item Parameters such as wavelengths are controllable.
\item For planar domains with small obstacles, the wavefronts behave as if there are no obstacles at all (see Figure \ref{fig:ghm in square}): waves propagate through, making the problem of undersampling ignorable.
\end{enumerate}

The present paper begins where \cite{gha} ends, by investigating what happens when this protocol is adapted to an {\em indoor} network where the geometry and topology is not that of an open plane (perhaps dotted with obstacles), but rather a system of fairly narrow hallways connected with a non-trivial topology: a ``fat'' planar graph. The contributions of this paper include the following:
\begin{enumerate}
\item We observe and then prove that wavefronts propagate through the hallways, turning corners and branching off to side-corridors.
\item We lift the CCA dynamics from the network to the higher-dimensional simplicial complex the network bounds.
\item We detail a pursuit-evasion game within the domain and give sufficient conditions (in terms of the topological features of the system) for the pursuer to win.
\item We show how wavefronts have well-defined cohomology classes and prove a realization theorem for which cohomology classes can be attained by the system.
\item We identify what we believe is a novel type of {\em global} defect in CCA, generated by the topology of the domain as opposed to a local singularity. We show how such \style{pulse solutions} behave like solitons in the system.
\end{enumerate}

Most existing work on conserving energy for WSN focuses on distributed sleep-wake scheduling. For example, PEAS\cite{peas} provides a protocol by forcing a node who has an active neighbor to sleep for period according to exponential distribution. It is robust against node failure, however, could not guarantee, or measure the coverage with the rapid change of active sensors. The CDSWS\cite{cdsws} protocol uses a clustering technique to divide the sensors into multiple clusters, and selects a few sensors from each cluster to work, while maintaining nearly full coverage. ASCENT\cite{ascent} allows sensors to measure their connectivity in the network in order to activate their neighbors based on those measurements. But it never allows working sensors to go back to sleep again, which ends up consuming more energy as time goes by. Compared to those works, our protocol provides the ``user'' a chance to determine how much energy they would allow to be consumed, as balanced against how long it takes the system to detect evaders in the environment. The more energy it consumes, the less the expected time would be. Another advantage over the other protocols is it guarantees the failure of any evader following continuous path in the domain. Although our scheme requires synchronization ahead of time, and has not taken into account node and link failure yet, it provides a new approach to designing distributed sleep-wake WSN with energy constraints.

The outline of our paper is as follows.
\S\ref{sec:GHM} provides a network protocol, along with simulations and observations. In \S\ref{sec:asymptotic}, we first introduce the topological tools used later, then classify the asymptotic behavior of the system, and detail a necessary and sufficient condition for the system to not converge to an all-wake state. Degree as a time invariant is first introduced. We also formally define the ``evasion game'' at the end of this section. \S\ref{sec:limitcase} verifies the system on the one-dimensional limiting space defined by the hallways. For proofs of normal hallways case as in the simulations, please refer to \S\ref{sec:proof}, where we turn to topological tools. \S\ref{sec:seedoff} is a supplementary section, answering questions such as what the dynamics would be if there is no ``local defect'', by paring degree and cohomology. Next comes \S\ref{sec:link failure}, which briefly analyzes the system's feasibility even under link failures. Conclusion and comments are in \S\ref{sec:conclusion}.

\section{Greenberg-Hastings Model and Simulation}
\label{sec:GHM}

\subsection{Cellular Automata}
\label{subsec:CA}

A \style{cellular automata}, (CA), is a lattice-space and discrete-time dynamical system. Spatial coordinates are called \style{nodes}, and the dynamics generally take values in a finite alphabet $\Alphabet$, with $\Alphabet=\zed_2=\{0,1\}$ being the most common choice. The dynamics are local, in that the update rule for a node is a function of its state and the states of its spatial neighbors. For example, in $\zed^{2}$ lattice, the Von Neumann neighborhood of a node with coordinate $(i, j)$ is defined as the set of nodes attached to it, including itself, \ie, $\{(i-1, j), (i, j-1), (i+1,j), (i, j + 1), (i, j)\}$. An \style{initial state} (time $t = 0$) is selected by assigning a state for each node, typically at random. A new generation is created (advancing $t$ by $1$), according to some fixed rule universally that determines the new state of each node in terms of the current states of the node and its neighborhood. Typically, the updating rule is the same for each node and does not change over time, and is applied to the whole space simultaneously (but see asynchronous cellular automata \cite{gha} for one exception).

This paper focuses exclusively on \style{cyclic} cellular automata(CCA). The alphabet is defined to be $\Alphabet=\zed_{n} = \{0,1,\dots, n-1 \}$ under modular arithmetic. One denotes the (discrete) collection of nodes as $X$ and denotes a \style{state} at time $t$ as $u_t : X\rightarrow\zed_{n}$. The updating scheme for a general CCA is increments states in $\zed_n$, assuming that some excitation threshold is exceeded. More specifically, $u_{t+1}(x)=u_t(x)+1$ if certain criteria concerning the states of the immediate neighbors of $x$, $\Neighbor(x)$, are met. Such systems tend to cause periodic or cyclic behavior, spatially distributed and organizing into waves.

The \style{Greenberg-Hastings Model (GHM)} is a CCA first invented to study the spatial patterns in excitable media \cite{greenberg}. In this model, special significance is given to a single state (state 0), interpreted as an excitation state. The update rule for GHM is as follows:

\begin{displaymath}
u_{t+1}(x) =\left\{
\begin{array}{cll}
u_t(x) + 1 & : & u_t(x)\neq 0 \\
1 & : & u_t(x)=0 \, ;\, u_t(y)=1 \; {\mbox{ for some }} \;y\in\Neighbor(x) \\
0 & : & {\mbox{ else }}
\end{array}\right.
\end{displaymath}

This updating scheme is interpreted as the result of two mechanisms combined, excitation mechanism diffusion. It is therefore no surprise that there is a strong resemblance between the behavior of GHM on the plane and solutions to reaction-diffusion PDEs on planar domains \cite{pde_greenberg}, with both generating spiral-type waves. Given a fixed network, the states of the nodes will be uniquely determined by the initial state, because the GHM is a deterministic model. Denote by $\Update$ the evolution operator $\Update:\Statespace\to\Statespace$.

\subsection{Observations}
\label{subsec:observation}

Figure \ref{fig:experiment} illustrates the dynamics of the GHM on a specific indoor network. The network is built on \style{narrow hallways} modeled as a metric space with Euclidean metric; the neighborhood of a node is defined as nodes within distance $r$. We parameterize the system of 16250 nodes inside a $200\times200$ square with $n = 20$ and $r = 1.5$. The colors are representing states, with dark blue representing state 0. At time=0, it is in an unordered initial state. During the first several time steps, generally until time 25, the ratio of nodes with states 0 grows, as a result of the fact that states grow steadily until they reach 0 and wait for a stimulation from its neighborhood. At around time 45, spiral patterns become clearer visually, from top left, bottom and middle right. Those spiral ``seeds'' propagate waves along hallways. \style{Wavefronts}, consisting of the nodes with state 0, sweep through the domain, traveling ``intelligently,'' turning corners, etc. When wavefronts coming from different directions meet, they annihilate. And after enough steps (about 250), wavefronts have finally covered the whole space.

\begin{figure}
\centering
\begin{tabular}{cc}
\epsfig{file=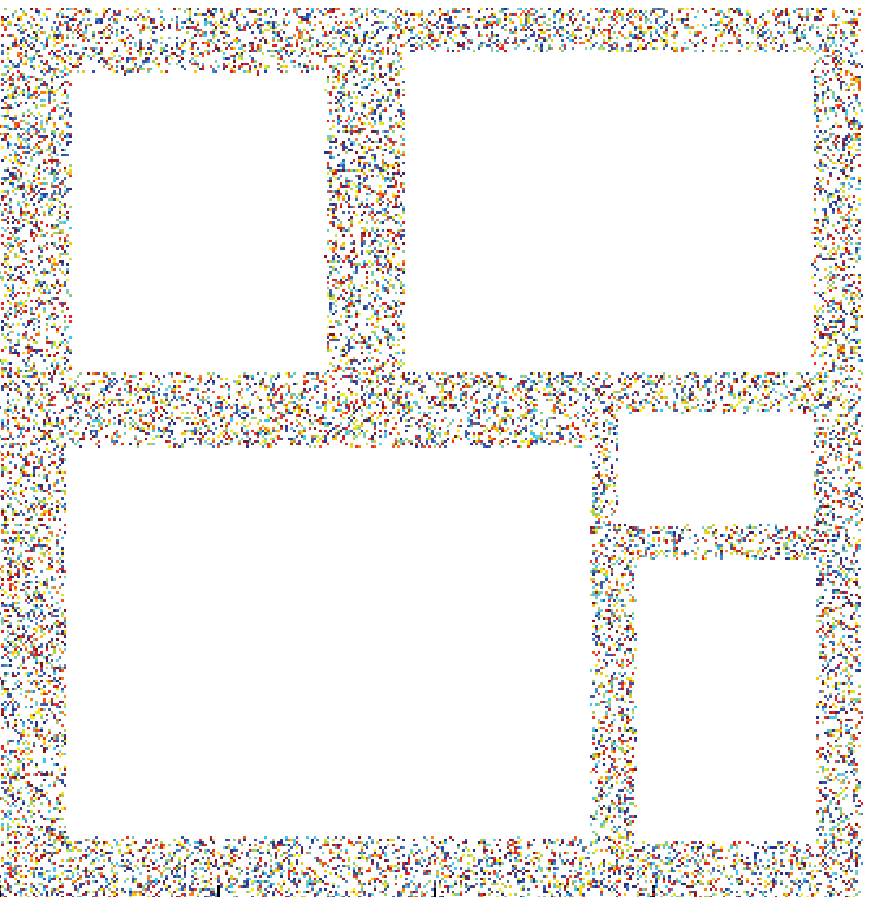,height=1.7in,width=1.8in} &
\epsfig{file=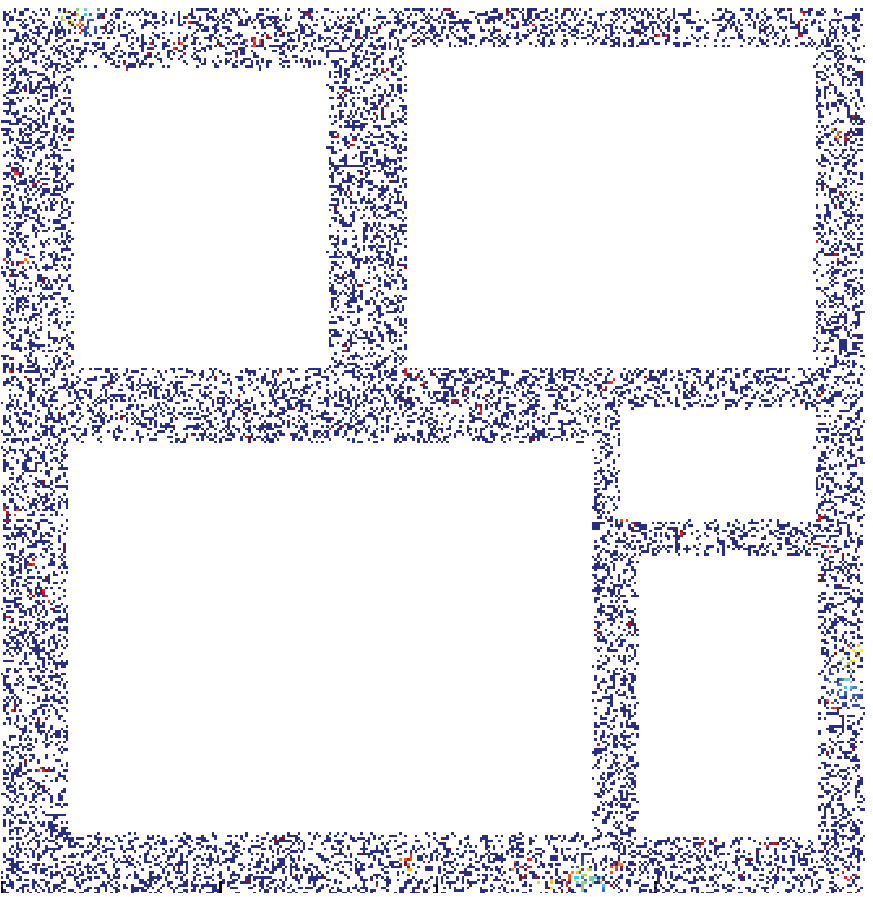,height=1.7in,width=1.8in} \\
\epsfig{file=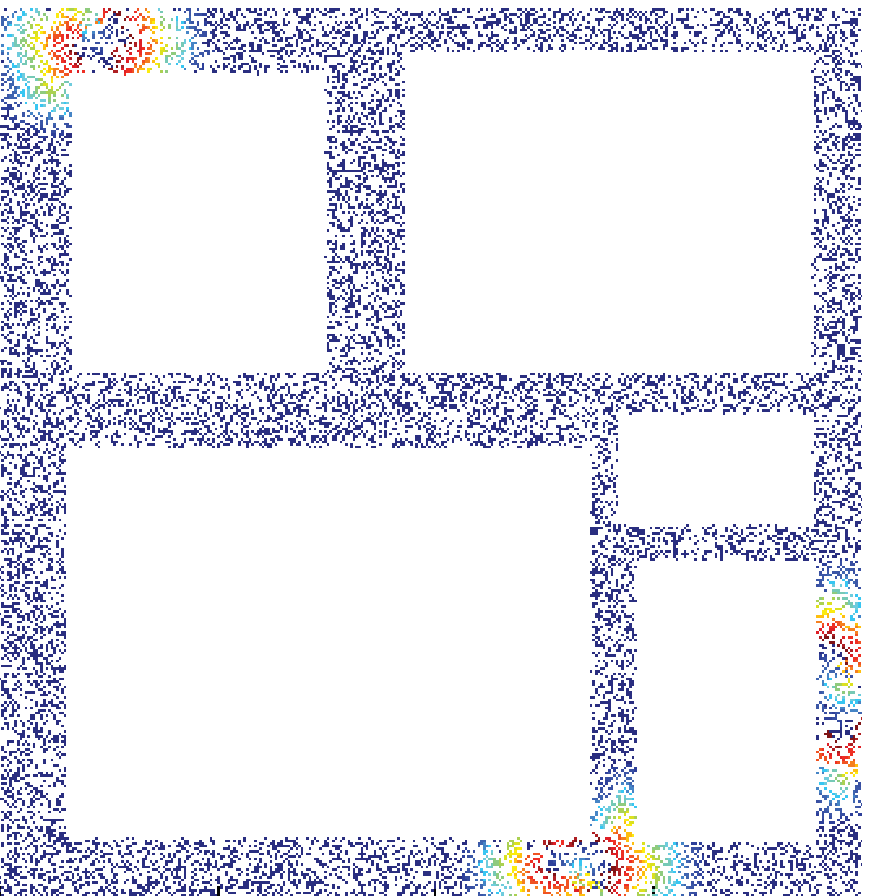,height=1.7in,width=1.8in} &
\epsfig{file=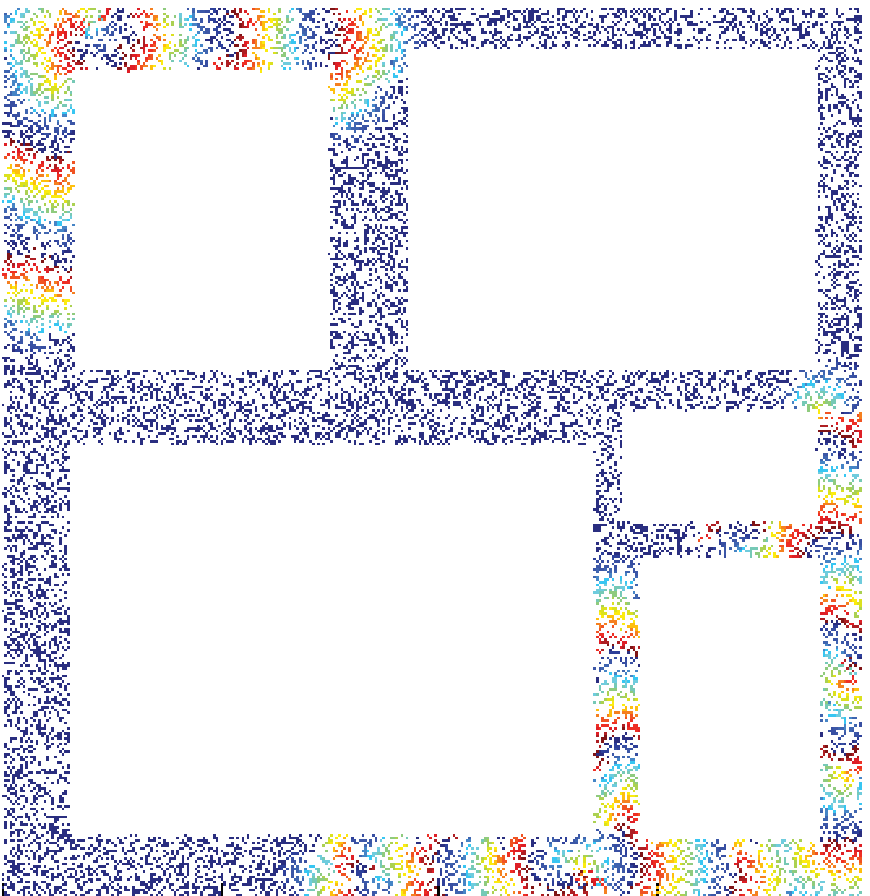,height=1.7in,width=1.8in} \\
\epsfig{file=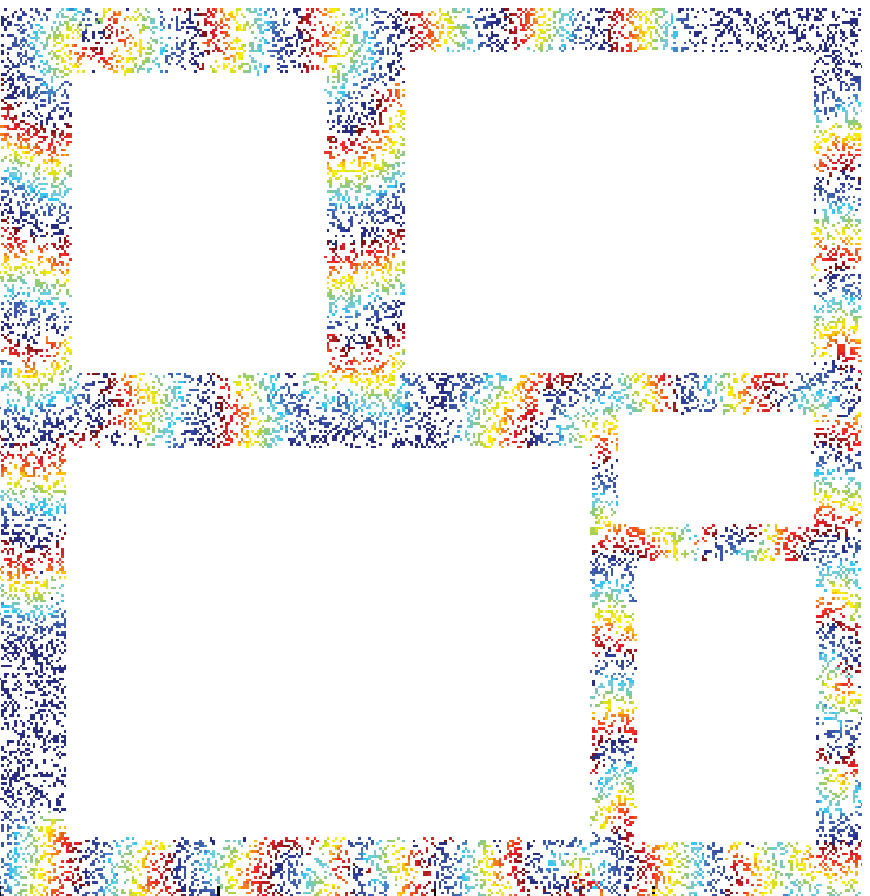,height=1.7in,width=1.8in} &
\epsfig{file=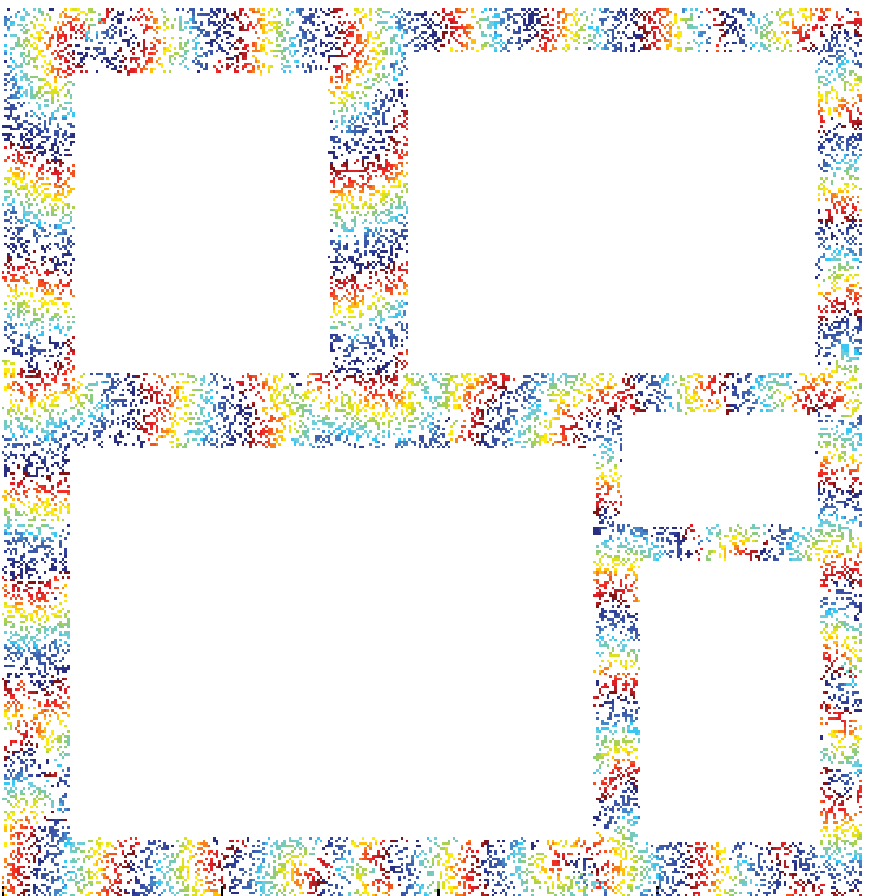,height=1.7in,width=1.8in} \\
\epsfig{file=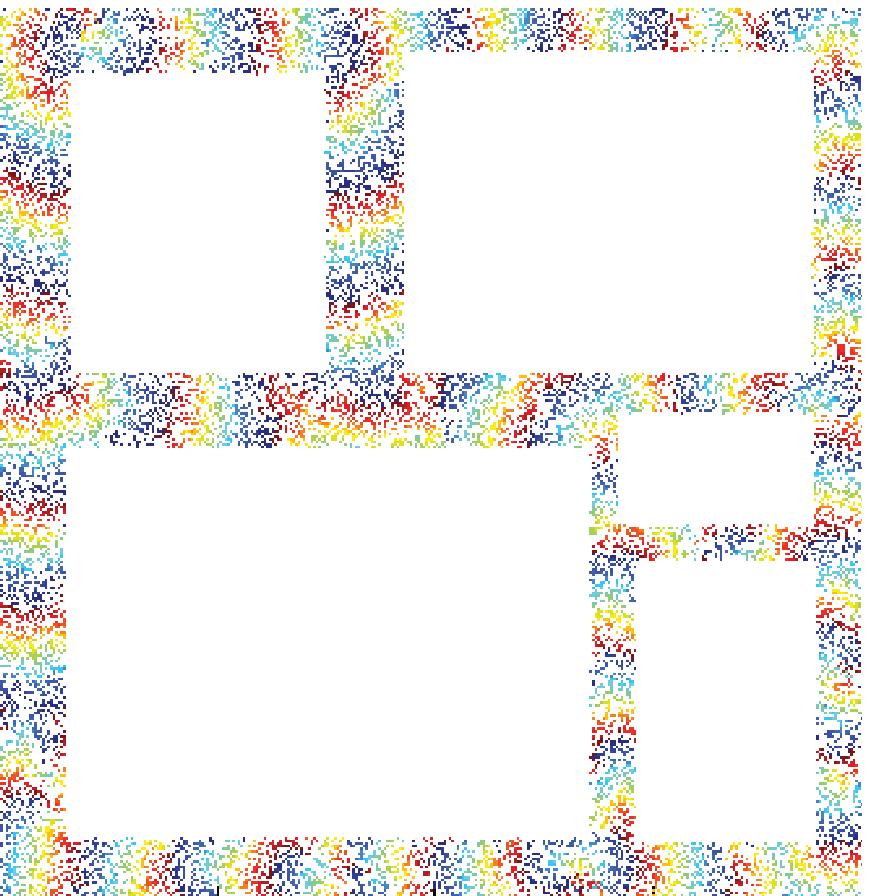,height=1.7in,width=1.8in} &
\epsfig{file=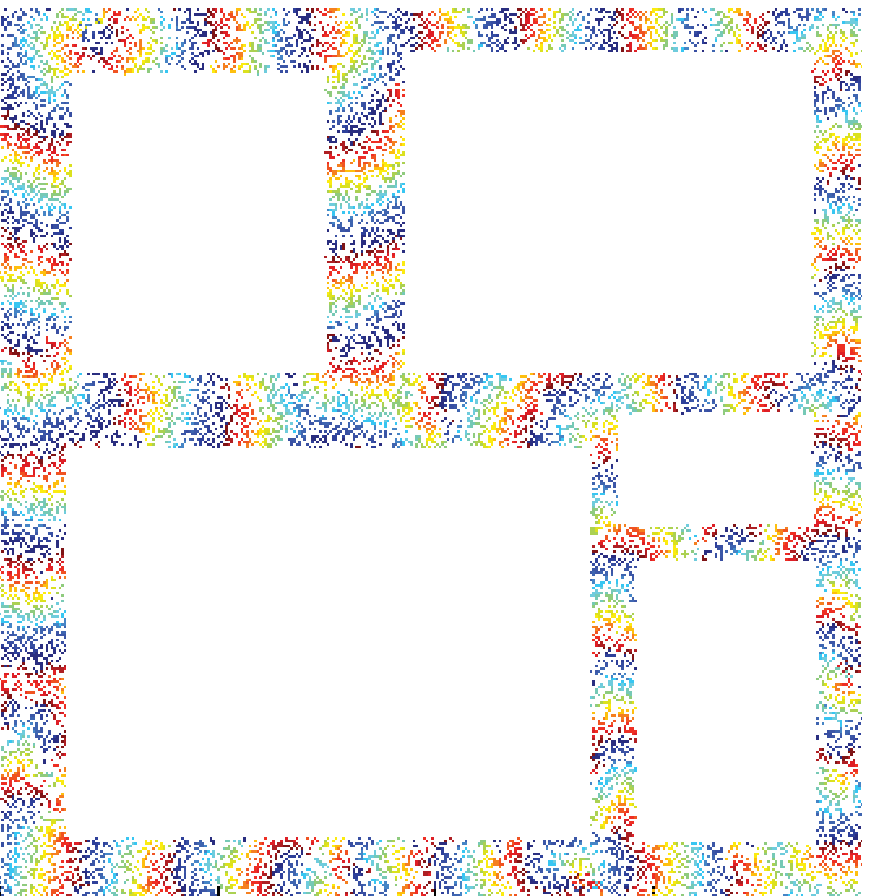,height=1.7in,width=1.8in}
\end{tabular}
\caption{Greenberg-Hastings Model for Narrow hallways space at time 0, 20, 45, 90, 150, 200, 250, 350.}
\label{fig:experiment}
\end{figure}

This protocol has some properties that make it ideal as a intruder detection sensor network.

\begin{enumerate}

\item The system has tunable energy efficiency. Set state $0$ to be the waking state, with all other states $2, 3,\dots, n-1$ denoting sleep mode. Those sleep-mode sensors are doing nothing but advancing their states by 1 for every time step. This could be done with very low energy consumption because they only have to follow clock clicks with no computing, sensing, or transmitting. Intrusion-detection is performed by the wake-state $0$ nodes. After a sufficient time lapse,  only a fraction (about $1/n$) of the total sensors will be in wake mode at any given time: the larger $n$, the less energy consumed.

\item The wave length is generally fixed no matter which source it is from, assuming the nodes are uniformly distributed: it appears to depend linearly on $n$. For bigger $n$, longer wave length is generated, but seeds are generated with smaller probability and longer generation time. This makes a trade-off between energy consumption and system success.

\item If we are given some specific sensors, say, with a big enough sensing radius $\epsilon$, then we will see that nodes in wavefronts form barriers, cutting the hallway into disconnected pieces. We note that the wavefronts efficiently sweep the corridors. Any intruder between two barriers has to follow the direction wavefronts propagate in order not to be detected immediately, but still is not able to survive to the end and will be detected by a coming wavefront in the opposite direction.

\end{enumerate}

\section{Topological tools and dynamical features}
\label{sec:asymptotic}

\subsection{Topological tools}

Our goal of a decentralized protocol for integrating local data into a comprehensive understanding of the global system points to algebraic topology as an appropriate and useful toolset. 
In order to construct a topological object based on the network that preserves local information, we consider the \style{flag complex} of the network graph. Recall, a \style{simplicial complex} is a union of simplices obtained by gluing them together along faces of same dimension \cite{hatcher}. The flag complex (also known as a \style{clique complex}) of a network is the maximal simplicial complex with the network graph as 1-skeleton.

\begin{definition}
Given an undirected network graph $G$, with vertex set $X$ and edge set $E$, the flag complex $C_f$, of $G$ is the abstract simplicial complex whose k-simplices correspond to unordered $(k+1)$-tuples of vertices in $V$ which are pairwise connected by edges in $E$.
\end{definition}

If metric information about the network is given, the \style{Vietoris-Rips complex} can also be built.

\begin{definition}
Given a set of points $X$ in a metric space and a fixed parameter $r >0$, the Vietoris-Rips complex of $X$, $\Rips_{r}(X)$, is the abstract simplicial complex whose $k$-simplices correspond to unordered (k + 1)-tuples of points in $X$ which are pairwise within distance $r$ of each other.
\end{definition}

Compared to a flag complex, the (Vietoris-)Rips complex requires metric information about the space. But if the network on a metric space is built with communication radius $r$, which means two nodes are neighbors if and only if they are with in distance $r$, then the flag complex and Rips complex built on this network are exactly the same.



\begin{definition}
\label{def:shadow}
For an abstract simplicial complex $C$ whose 0-simplices are located in a d-dimensional Euclidean space $\euc^d$, the shadow of $C$ in $\Domain$, $S(C)$, is the union of convex hulls of each simplex in $C$.
\end{definition}

\subsection{Dynamical features}

We reprove certain results from the CCA literature \cite{cca_in2d,durrett} in the more general setting of network (as opposed to lattice) systems. Our perspective is that a CCA is a discrete-time network-based dynamical system. From observation, the interesting dynamical features associated with the GHM are time-periodic. We therefore focus our efforts on understanding time-periodic states.

\begin{definition}
\label{def:periodic}
An \style{orbit} of a node $x\in X$ under GHM with an initial state $u_0$ is the time-sequence of states $(u_t(x))_{t=0}^\infty$. A node $x$ is said to be \style{$K$-periodic} if its orbit satisfies $u_{t+K}(x)=u_{t}(x)$ for some $K>0$ and all $t$. A node $x$ is said to be \style{eventually periodic} if its orbit satisfies $u_{t+K}(x)=u_{t}(x)$ for some $K>0$ and all sufficiently large $t$.
\end{definition}

\begin{definition}
A state $u$ on a subgraph $X'\subset X$ is \style{continuous} if for every pair of neighbors $x,y\in X'$, $|u(x)-u(y)|\leq 1$ (where, recall, all addition is in $\zed_n$).
\end{definition}

\begin{definition}
A node $x$ is \style{subordinate to} 
a neighbor $y$ at time $t$ if their states at that time satisfy $u_t(y) = u_t(x) + 1$ (where, recall, all addition is in $\zed_n$).
\end{definition}

\begin{lemma}
Subordinate nodes will remain continuous for all future time.
\end{lemma}

\begin{proof}
It suffices to assume a subgraph consisting of a single edge with nodes $x$ and $y$. Assume that $|u_t(x)-u_t(y)|\leq 1$. Consider the set $S=\{z\in X| u_{t+1}(z)=u_t(z)\}\subset u_t^{-1}(0)$. Depending on membership in $S$,
\begin{equation}
\label{eq:update diff cases}
\left(u_{t+1}(x) - u_{t+1}(y)\right)-\left(u_t(x) - u_t(y)\right)=
\begin{cases}
0& \text{$x,y\in S$ or $x,y\notin S$}\\
1& \text{$x\notin S$ and $y\in S$}\\
-1& \text{$x\in S$ and $y\notin S$}
\end{cases}
\end{equation}

By continuity, $|u_t(x)-u_t(y)|\leq 1$, so $|u_{t+1}(x)-u_{t+1}(y)|$ will exceed $1$ only if $u_t(x)-u_t(y)=1,((\Update(u))(x)-(\Update(u))(y))-(u(x)-u(y))=1$ or $u(x)-u(y)=-1,((\Update(u))(x)-(\Update(u))(y))-(u(x)-u(y))=-1$. The first case is equivalent to $u(x)=1, u(y)=0$ and $x\notin S, y\in S$, which is impossible because $y$ has neighbor $x$ in state $1$, and will not stay in state $0$ for the next step, thus not in $S$; the second case is the symmetric case which by the same argument is not possible either. Then $|(\Update(u))(x)-(\Update(u))(y))|\leq 1$, which makes $\Update(u)$ also continuous.
\end{proof}

\begin{corollary}
\label{cor:subordination hold}
If a node $x$ is subordinate to a neighbor $y$ that is $n$-periodic at time $t$, then node $x$ is $n$-periodic for all future time. Any node subordinate to an eventually periodic node is eventually periodic.
\end{corollary}

\begin{proof}
Suppose $x$ reaches $0$ for the first time (after $t$) at time $t_0$. By the scheme of $\Update$, $u_{t_{0}}(y)=1$. Therefore, all we need to prove is for any non-negative integer $k$, $u_{t_{0}+kn+1}(x)=1$. We have already proved $u_{t_{0}+1}(x)=1$ because it has a neighbor $y$ in state $1$ at that moment. Suppose the statement holds for a particular $k_0$, \ie, $u_{t_{0}+k_0n+1}(x)=1$, then $u_{t_{0}+(k_0+1)n}(x)=0$. But by periodicity, $u_{t_{0}+(k_0+1)n}(y)=1$, thus $u_{t_{0}+(k_0+1)n+1}(x)=1$, which makes the statement hold at $k_0+1$. By induction, the statement holds for all $k$.
\end{proof}

\begin{corollary}
\label{cor:continuity holds}
 Continuity is \style{forward-invariant}: continuous states remains continuous in time.
\end{corollary}
\begin{proof}
According to Corollary \ref{cor:subordination hold}, two neighbors that are subordination will remain continuous. For one step forward, two neighbors that are of the same state will either remain the same state, or be offset by state 1, which means subordination, thus also continuous.
\end{proof}

However, it is {\em not} necessarily the case that all initial conditions converge to a continuous state (even in a connected compact network). See, for example, Figure \ref{fig:non-continuous state}: every node has period $n$, and the two nodes on the right end have states always differed by $3$. Thus this is never a continuous network.
\begin{figure}
\centering
\epsfig{file=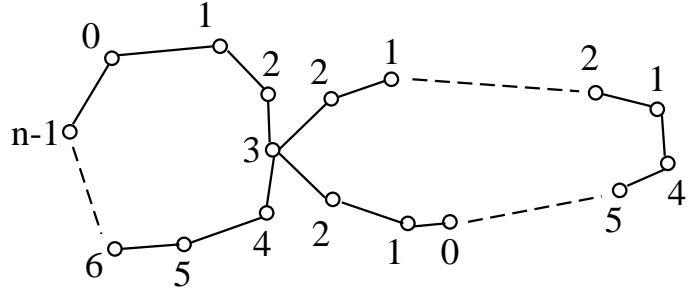}
\caption{Counter example: non-continuous state for all time.}
\label{fig:non-continuous state}
\end{figure}

The following definition is a network-theoretic version of the lattice-based analogue from, \eg, \cite{cca_in2d}.
\begin{definition}
\label{def:cycle}
We call a formal linear combination $\alpha$ of edges $\alpha_i=[a_i, b_i], i=1,\dots,K$ a \style{cycle} if the boundary of $\alpha$, $\partial\alpha=\sum_{i=1}^{K}(b_i-a_i)$ is 0. A cycle is called a \style{loop} if $b_i=a_i+1$ for $i=1,\dots,K-1$, and $b_K=a_1$.
\end{definition}

As a remark, a loop is a cycle, and a cycle is the sum of one or more loops. We also remark that the set of cycles $Z$ has the structure of an abelian group: one can add cycles and scale them by (integer) coefficients.

\begin{definition}
\label{def:seed}
A state $u:X\to\zed_n$ has a \style{seed} if there is a loop $\sum_{i=0}^{K-1}[x_i,x_{i+1}]$ ($x_0=x_K$), for which $u(x_i)=i \mod n$.
\end{definition}

\begin{figure}
\centering
\begin{tabular}{cc}
\epsfig{file=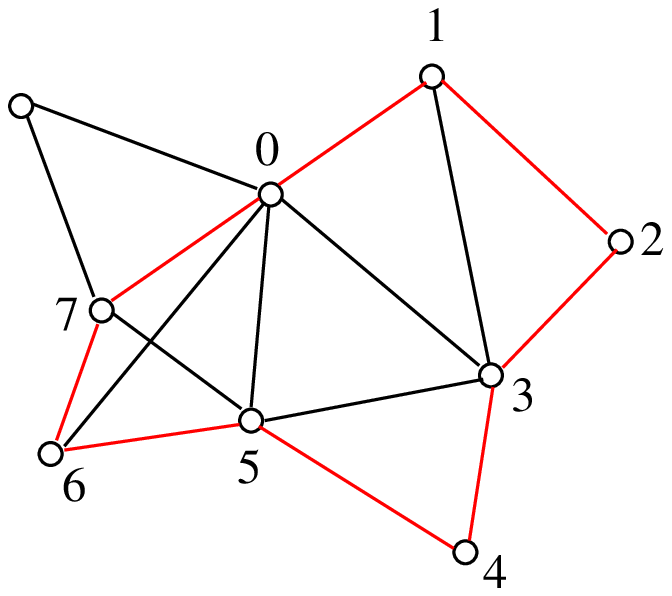,height=2in,width=2in} &
\epsfig{file=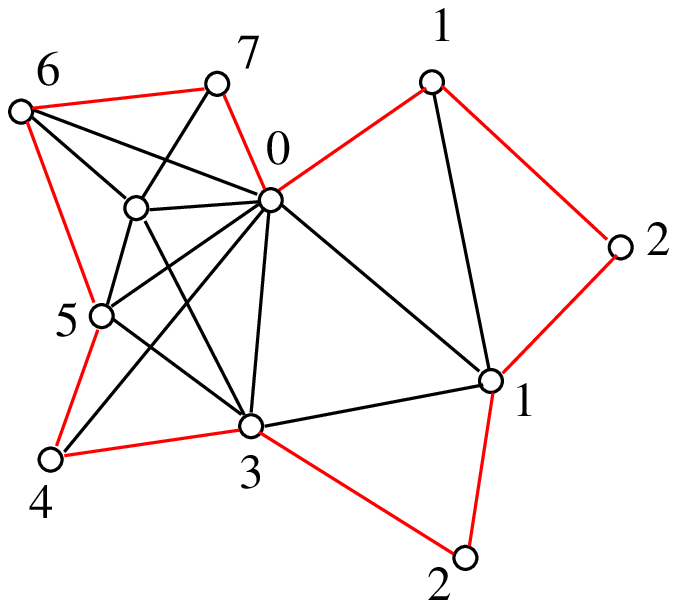,height=2in,width=2in}
\end{tabular}
\caption{A seed (left) and a defect (right) for $n=8$ on cycles in light red.}
\label{fig:seed and defect}
\end{figure}

By definition, the length $K$ of a loop that makes a seed has to be a nonzero multiple of $n$, because $u(x_0)=u(x_K), K=0 \mod n$.
Every node on a seed has period $n$, because it will have a neighbor of state $1$ on the seed when it reaches state $0$.

\begin{lemma}
If an initial condition $u_0$ on a connected compact network $X$ contains at least one seed, then all nodes are eventually periodic.
\end{lemma}
\begin{proof}
Let the set of $n$-periodic nodes in $X$ be $P_t$, and state at time $t$ be $u_t$. Suppose the loop $\sum_{i=0}^{K-1}[x_i,x_{i+1}], x_0=x_K$ makes a seed in initial condition, then $P_0$ is nonempty, with $(x_i)_0^{k-1}$ as a subset. $P_t$ is non-descending with respect to time $t$, $P_0\subset P_1\subset\dots\subset P_t\subset P_{t+1}\subset\dots$. For a node $x$ that is not $n$-periodic that has at least one neighbor that is $n$-periodic at time $t_0$, if $x$ never gets to be $n$-periodic, it means for any $t$ positive, there exists some $s\geq t$ such that $u_{s+1}(x)\neq u_s(x)+1 \mod n$. It will induce that $u_{s+1}(x)=u_s(x)=0$, which is saying $x$ gets to stay in state 0 for a while from time to time. But the neighbors of $x$ that are periodic $n$ are advancing their states by 1 at every time step, this will make the face difference between them and $x$ bigger and bigger until it reaches $1 \mod n$. When the such a offset by 1 appears, a subordination between $x$ and its periodic $n$ neighbor is built up, which makes $x$ periodic ever since as a result of corollary \ref{cor:subordination hold}. Therefore every node which as at least one neighbor that is periodic $n$ will be periodic $n$ after a finite amount of time (no longer than $n$). By the above argument and the fact that the network is connected, for any $t$, $P_t\subsetneq P_{t+n}$. On the other hand side, since $X$ is compact, there exist a time $T$, such that
\begin{displaymath}
\bigcup_{t=0}^{\infty}P_t=P_T
\end{displaymath}
Therefore, the whole system is in $n$-periodic state since time $T$.
\end{proof}

We see in the above arguments that a loop that makes a seed at one moment will support a seed forever with the dynamics. The key feature that is invariant under the dynamics is the concept of ``winding number'', which records how many rounds it goes through while chasing continuously on a loop. We will define this as degree and extend the concept to all cycles.

\begin{definition}
\label{def:degree}
For a given network $X$ and a state $u\in\Statespace$, if $u$ is continuous on a cycle $\alpha=\sum_{i=1}^{K}[a_i,b_i]$, then the \style{degree} of $u$ on this cycle is defined as
\begin{displaymath}
\degree(u,\alpha):=1/n\sum_{i=0}^{k-1}(u(b_i)-u(a_i))
\end{displaymath}
where the summands are forced to be $-1$, $0$, or $1$, and the sum is ordinary addition (not $\mod n$).
\end{definition}

\begin{definition}
\label{def:defect}
We call a cycle $\alpha=\sum_{i=1}^{K}[a_i,b_i]$ in the network $X$ a \style{defect} for some state $u\in\Statespace$ if the degree of $u$ on this cycle is nonzero.
\end{definition}

An example of a defect is as in Figure \ref{fig:seed and defect}. The concept of a defect is a generalization of a seed, in the sense that it has nonzero degree. The term ``degree'' defined here is consistant with the use of degree in topology, which is a homotopy invariant \cite{hatcher}. Here, it is the discrete version of ``winding number'' for continuous self-maps of the circle $S^1$ \cite{durrett}, describing how many times it wraps around with direction. Similar to Lemma 5 in \cite{durrett}, we will prove the $\real^2$ version instead of the lattice $\zed^2$ version, presenting a necessary and sufficient condition for a continuous system not dying out.

\begin{lemma}
\label{lem:degree of summation}
For two cycles $\alpha$ and $\beta$, if a state $u$ is continuous on both cycles, then it is also continuous on their sum $\alpha+\beta$, and $\degree(u,\alpha+\beta)=\degree(u,\alpha)+\degree(u,\beta)$.
\end{lemma}
\begin{proof}
Let $\alpha=\sum_{i=1}^{K}[a_i,b_i]$ and $\beta=\sum_{i=1}^{L}[c_i,d_i]$, then
\begin{align}
&\degree(u,\alpha+\beta)\\
=&1/n(\sum_{i=0}^{k-1}(u(b_i)-u(a_i))+\sum_{i=0}^{k-1}(u(d_i)-u(c_i)))\\
=&1/n\sum_{i=0}^{k-1}(u(b_i)-u(a_i))+1/n\sum_{i=0}^{k-1}(u(d_i)-u(c_i))\\
=&\degree(u,\alpha)+\degree(u,\beta)
\end{align}
\end{proof}

\begin{lemma}
\label{lem:invariant degree}
For a cycle $\alpha$ and a continuous state $u$, the degree of $u$ on this cycle is invariant under the GHM updating rule $\Update$, i.e.,
\begin{equation}
\label{eq:update difference}
1/n\sum_{i=0}^{k-1}((\Update(u))(x_{i+1})-(\Update(u))(x_i))=1/n\sum_{i=0}^{k-1}(u(x_{i+1})-u(x_i))
\end{equation}
\end{lemma}
\begin{proof}
We first prove that degree on a loop $\sum_{i=0}^{K-1}[x_i,x_{i+1}], x_0=x_K$ is invariant. As before, equation \ref{eq:update diff cases} holds for every pair of neighbors $x_{i+1}$ and $x_i$. Since the number of pairs $(x_{i+1},x_i)$ with $x_i\in S, x_{i+1}\notin S$ is the same as the number of pairs with $x_i\notin S, x_{i+1}\in S$, the summation of $((\Update(u))(x_{i+1})-(\Update(u))(x_i))-(u(x_{i+1})-u(x_i))$ is $0$, which makes Equation \ref{eq:update difference} hold. Since every cycle is the sum of one or more loops, and degree is additive by Lemma \ref{lem:degree of summation}, then it is also invariant on cycles.
\end{proof}

For a state $u$ on a loop that forms a defect, if the loop bounds a region $V$ in $\real^2$ that belongs to $\Domain$, we can discuss the continuity of the subnetwork in $V$. If the subnetwork in $V$ is sufficiently dense (\eg, 2-complex has shadow containing $V$), we observe that the subnetwork could never reach continuity, with at least one singularity (a discontinuity) forced, as in Figure \ref{fig:defect discontinuity}. This could be understood intuitively as a discrete version of the theorem in complex analysis, which says a holomorphic function on a domain always has integration $0$ on the boundary. It would also contradict the fact that a continuous map from a contractible space to $S^1$ has degree $0$ restricted on any loop. 

\begin{figure}
\centering
\epsfig{file=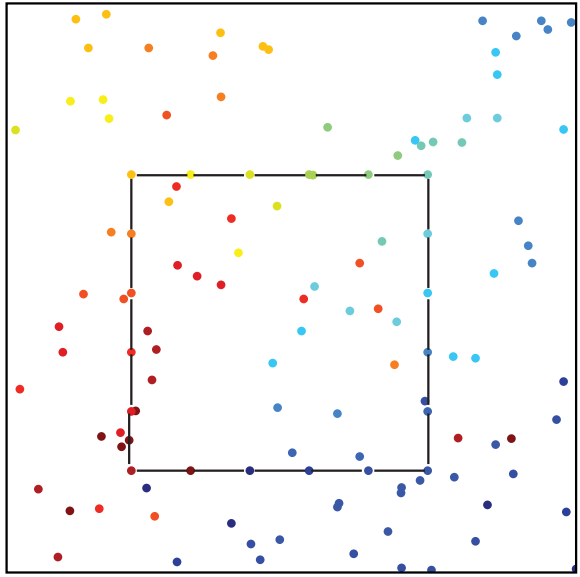,height=3in,width=3in}
\caption{In the region bounded by a defect, the state is discontinuous.}
\label{fig:defect discontinuity}
\end{figure}

\begin{lemma}
\label{lem:defect discontinuity}
Consider a state $u$ on $X$ with $n>3$, and a loop $l=\sum_{i=0}^{K-1}[x_i,x_{i+1}], x_0=x_K$ in $X$. If the loop $l$ is null homologous in the 2-complex built on $X$, and $u$ on $l$ makes a defect, then $u$ is discontinuous on $X$. For $n\leq 3$ any state on $X$ is continuous.
\end{lemma}

\begin{proof}
Since $l$ is null homologous in 2-complex built on $X$, $l=\partial\beta$ where $\beta=\sum_{i=1}^{m}\beta_{i}$ is a 2-chain in the 2-complex and $\beta_{i}$ are 2-simplices. $l$ can be deformed to a single 2-simplex through a sequence of homologous loops in $X$, $\sum_{i=1}^{m-j}\partial\beta_{i}, j=1,\dots,m-1$, while the successive two loops only differ by the boundary of one 2-simplex. Suppose $u$ is continuous on $X$, such operation could not change the degree at all, since at most three of the summands $u(x_{i+1})-u(x_i)$ has been changed value up to 1. Thus the summation is at most changed by 3, which makes the degree changed by at most $3/n$, which has to be invariant when $n>3$. Thus the nonzero degree remains the same for the sequence of homologous loops, which can not be true because the degree on the boundary of a single 2-simplex has to be zero. Therefore the state on $X$ could not be continuous.

It is trivial to see the continuity when $n\leq 3$, because any two elements in $\zed_n$ differ by at most 1.
\end{proof}

\begin{theorem}
\label{thm:die iff no defect}
For a continuous state $u$ on a connected compact network $X$, the system eventually turns to all-$0$ state (die out) if and only if $u$ does not contain a defect.
\end{theorem}

\begin{proof}
Suppose $u$ contains a defect on cycle $\alpha=\sum_{i=1}^{K}[a_i,b_i]$.
By lemma \ref{lem:invariant degree}, the degree is invariant under $\Update$, so it will never be $0$, thus the system will never turn to all-$0$ state.

For the converse, we need to show for a continuous state $u$ not dying out eventually, it has to contain a defect at the beginning. Firstly, it is obvious that after long enough time, in such system $u_t$, every node in state $i$ must have a neighbor in state $i+1$, for all $i\neq0$, since $u_t=\Update^i(u_{t-i})$. So if we start from a node $x_0$, such that $u_t(x_0)=1$, we can find a neighbor of $x_0$, $x_1$, such that $u(x_1)=2$. Following the process, we get a sequence of nodes $x_0,x_1,\dots,x_{n-1}$, such that $x_i$ and $x_{i+1}$ are neighbors and $u_t(x_i)=i+1 \mod n$. If from every state $0$ node, a state $1$ node could be reached by jumping along neighbors which are in state $0$, then following the process, we will finally reach a node has been visited before, as $X$ is compact. In this way, we have obtained a defect in $u_t$, which is also a defect in $u$, by lemma \ref{lem:invariant degree} and the fact that $u$ is continuous.

To see that a node with state $1$ could always be reached from a node with state $0$ by jumping along neighbors in state $0$, all we need to prove is there is no such set $A$ of nodes with state $0$, that their neighbors not in $A$ could only be in state $n-1$. If such $A$ exists in $u_t$, then there is a proper subset of $A$ with state $0$, and state $n-1$ on the complement in $u_{t-1}$. Following these procedure, we should finally obtain a set $A_0$ of state $0$ nodes, each has at least one neighbor with state $n-1$ and other neighbors with state $0$ in $u_{s}$. Then in $u_{s-1}$, nodes in $A_0$ have to be in state $n-1$ (by continuity), and their neighbors not in $A_0$ must all be in state $n-2$, and for $n$ steps back, in $u_{s-n}$, nodes in $A_0$ have to be in state $0$, and their neighbors not in $A_0$ must all be in state $n-1$. But such a $u_{s-n}$ could not produce $u_{s-n+1}$ under $\Update$, because those state $0$ nodes have no neighbors in state $1$. Therefore such a set $A$ does not exist, which makes the statement in the beginning true.
\end{proof}

Since degree is invariant under the update rule $\Update$, Theorem \ref{thm:die iff no defect} can be interpreted as saying that a continuous state dies out eventually if and only if it is cohomologically trivial (see \S\ref{sec:seedoff} for details on how to define the cohomology class of a state).

\begin{theorem}
There exists a directed subgraph $\mathcal{F}$ of the network that is a spanning forest rooted at seeds, with directed edges in the direction of subordination.
\end{theorem}

\begin{proof}
Every node that is not originally a seed node will become periodic by building up a subordination with some periodic neighbor. For every non-seed node, choose one from its neighbors via subordination and use a directed edge with itself as head to represent the relationship. This forms a directed subgraph of the network. From any non-seed node, following those directed edges with inverse directions, it has to end in a seed node, because it is a compact network. We argue that the subgraph is a tree because it contains no loop: if it did contain a loop, then the loop is comprised of non-seed nodes, but for any directed edge, the head node becomes periodic later than the tail node, which is a contradiction with being in a loop. And furthermore, we can treat the forests as rooted at seeds, which makes every edge in the direction that goes deeper in a branch to the leaves. Such structure gives the nodes a hierarchy, and since for every edge, the two end nodes have states offset by 1, we can induce the state after the system reaches equilibrium.
\end{proof}

\begin{definition}
\label{def:depth}
The depth of a node in a tree is the number of hops between the node an the root of the tree.
\end{definition}

According to the above proof, we have made a point in that the growth of the forest is at most one level at a time, which means in every time step, there could not be more than one node from a same branch that becomes subordinated.

Starting from a uniformly randomly generated initial condition (a reasonable if idealized statistical model) the system is not guaranteed to converge to a periodic system, not including all-zero states. One sufficient condition is the existence of a seed, which we prove to be of high probability with certain reasonable assumptions (Lemma \ref{lem:seed prob}). It is possible that the system became messy with no wavefront observable (too many seeds all around in the space, for instance). We would require those nodes that are far away (in the hop-metric) from the defects to be in state 0 at one moment (in our case, larger than the number of states is already enough). This assumption is proved later to be of high probability (Lemma \ref{lem:far nodes die}). Under the above two assumptions, continuity in the acquired region will be guaranteed. Therefore, from now on, we limit discussion to the region far away from defects.

\begin{lemma}
\label{lem:seed prob}
For a fixed uniformly sampled network with communication radius $r$ and fixed $n$ on a domain consisting of fixed narrow hallways, the probability of at least one seed existing in the initial condition generated according to uniform distribution approaches 1 as the number of nodes grows.
\end{lemma}
\begin{proof}
Divide the space into square shaped pieces $\Domain_i$ indexed by $I$, with side length smaller or equal to $r/\sqrt{2}$. As the network size $|X|$ approaches infinity, the probability of there to be no less than $n$ nodes in each $\Domain_i$ approaches 1. For a $\Domain_i$ with $n$ or more nodes, the sub-network in this subdomain makes a complete graph. Therefore, there is no seed in the initial condition, if and only if the nodes do not cover all the states, which means there is at least one state missing in the initial condition. Thus
\begin{multline}
P(\text{no seed in initial condition in $\Domain_i$ with $m_i$ nodes})\\
\leq \frac{n(n-1)^{m_i}}{n^{m_i}}\\
= n(1-1/n)^{m_i}\\
\end{multline}
and
\begin{multline}
P(\text{no seed in initial condition in $\Domain$})\\
\leq \prod_{i=1}^{|I|}P(\text{no seed in initial condition in $\Domain_i$ with $m_i$ nodes})\\
\leq \prod_{i=1}^{|I|}n(1-1/n)^{m_i}\\
= n^{|I|}(1-1/n)^{|X|}\\
\end{multline}
which approaches 0 as $|X|$ approaches infinity.
\end{proof}

\begin{lemma}
\label{lem:far nodes die}
Starting with a fixed network and uniformly distributed initial conditions, with probability approaching $1$ as the state number $n$ grows, nodes with hop distance to all defects bigger than $2n$ will turn to state $0$ after $2n-2$ time steps.
\end{lemma}
\begin{proof}
Suppose there is no state $1$ node in $u_{n-1}$ in the region $n$ hops away from any defect, then there could be no state $1$ or $2$ node in $u_{n}$ in the region $n+1$ hops away from any defect, and with the same reason, there could only be state $0$ node in $u_{2n-2}$ in the region $2n$ hops away from any defect. Therefore the probability that every node at least $2n$ hops away from defects are state $0$ in $u_{2n-2}$ is no smaller than the that of no state $1$ node at least $n$ hops away from defects in $u_{n-1}$.

Now suppose there is a node $x$ at least $n$ hops away from any defect, and $u_{n-1}(x_1)=1$. Such $x_1$ must have at least one neighbor of state 2, named $x_2$, otherwise in one step before, it would not be able to update from $0$ to $1$. Via the same argument, there exists a sequence of nodes: $x_j, j=1,2,\dots,n-1$, such that $x_j$ and $x_{j+1}$ are neighbors, and $u_{n-1}(x_j)=j$. For one step ago, $u_{n-2}(x_j)=j-1$ for $j\neq n$, and two steps ago, $u_{n-3}(x_j)=j-2$ for $j=2,\dots,n-1$ and $u_{n-3}(x_1)\in\{0,n-1\}$. Following such argument, back at time $0$, $u_0(x_j)\in\{0,n-1,\dots,j+1\}$ for $j=1,\dots,n-2$ and $u_0(x_{n-1})=0$ with at least one neighbor of state $1$.

Let $I_j=\{0,n-1,\dots,j+1\}$. For a fixed node $x$,
\begin{multline}
P(\text{at least one of $x$'s neighbor have a state in $I_j$ at time 0})\\
=1-(1-j/n)^{|\Neighbor(x)|}\\
\end{multline}
where $|\Neighbor(x)|$ is the number of neighbors of node $x$. Suppose $\tilde{N}$ is a universal upper bound on $|\Neighbor(x)|$, then
\begin{multline}
P(u_{n-1}(x)=1\ \text{for some $x$ at least $n$ hops away from any seed})\\
\leq|X|\prod_{j=1}^{n-1}(1-(1-j/n)^{|\Neighbor(x)|})\\
\leq|X|\prod_{j=1}^{n-1}(1-(1-j/n)^{\tilde{N}})\\
\leq|X|(1-(1/2)^{\tilde{N}})^{(n-1)/2}\\
\end{multline}
which approaches $0$ as $n$ approaches infinity.
\end{proof}

For example, in a $40000$ nodes network, where every node could have up to 6 neighbors and $n=20$, the probability of a seed is bounded below by 0.9656, which validates the observation in previous simulation. As per the above two lemmas, we will always assume that at least one seed exist in initial condition, and the nodes at least $2n$ hops away from any defect will turn to state $0$ after $2n-2$ step. These two assumptions guarantee not only the system not dying out (turn into an all-zero state), but also the continuity of the system in acquired region, with the following corollary.


\subsection{Evasion Game}

We propose a sensor-network based ``Evasion Game'' formally, and then use the model to verify the system: why the wavefronts sweep the entire domain; how to interpret the phenomenon that wavefronts are dividing their neighborhoods and how to prove that an intruder will always fail to evading detection; what are the parameters that control the system and how they are changing the behaviors of those wavefronts.

\begin{definition}
\label{def:evasion game}
Let the domain where the evader and sensors are located be denoted $\mathcal{D}\subset\mathbb{R}^2$, and the sensor network $X$. For each sensor $x\in X$, its coverage is a subset $U_{x}\subset\mathcal{D}$. Denote by $X(t)$ the set of sensors in wake-state ($0$) at time $t$. We define the {\em Evasion Game} as follows: the strategy for the pursuer is to control the network following GHM, and the strategy for the evader is to pick a moment $t_0$ to come into the domain, and follow a continuous path in $\mathcal{D}$: $f: [t_0,\infty)\rightarrow\mathcal{D}$. The pursuer wins iff $\exists\ \tau\in[t_0,\infty)$, such that
\[
    f(\tau)\in\bigcup_{x\in X(\left\lfloor\tau\right\rfloor)}U_{x}.
\]
Otherwise, the evader wins.
\end{definition}

We note that the only requirement on the evader is its trajectory be continuous: there are no constraints on the velocity or acceleration. Even with such minimal constraints, the evader is not able to win.


\section{Limiting case with 1-d hallways}
\label{sec:limitcase}

We begin our analysis with the limiting case when every hallway is sufficiently narrow compared to the walls, so that the domain $\Domain$ can be approximated as a (topologically equivalent) one-dimensional space. We assume those sensors are located in $\Domain$ with each node having a coverage which is a one dimensional convex set around itself, and the convex hall of two neighbors is covered by the union of their coverage regions. We also assume the union of convex hulls of neighbors (subspace of $\Domain$) is good enough to cover $\Domain$, in which case the whole space is covered when every sensor is turned on. If we run GHM on this network, with at least one seed in initial condition, then every evader (not near the seeds) loses the evasion game.

\begin{theorem}
\label{thm:limit case theorem}
For GHM on network $X$ with communication distance $r$ in a compact and connected 1-d complex $\Domain$, if the initial condition contains at least one seed, and there exists a subnetwork $X'$ covering a sub-domain $\Domain'$, such that the state on $X'$ is eventually continuous and contains no defect, an evader will always lose the evasion game on $\Domain'$.
\end{theorem}
\begin{proof}
For any time $t_0$ when the evader comes into the domain, consider the product space $\Domain'\times[t_0,\infty)$ with the second coordinate representing time. Treat the coverage of the sensors also as a subspace $P_c$ of $\Domain'\times[t_0,\infty)$, which is
\[
    \bigcup_{t=\lceil t_0\rceil,t\in\zed}^{\infty}\bigcup_{x\in X'(t)}U_{x}\times[t,t+1)\ \ .
 \]
Let $p$ be the projection map: $p: \Domain'\times[t_0,\infty)\rightarrow\Domain', p(a,t)=a$. Then $p:P_c \rightarrow\Domain'$ is onto, because $\Domain'$ is fully covered when every node in $X'$ on. If we could prove that there exists a subspace in $P_c$ homeomorphic to $\Domain'$, with map $p$ as homeomorphism, then $\Domain'\times[t_0,\infty)\backslash P_c$ contains no continuous path from top $\Domain'\times\{t_0\}$ to bottom $\Domain'\times\{T\}$ for $T$ big enough, because they are dual to each other. Therefore no evader could survive forever. The construction of the subspace in $P_c$ is as follows in Lemma \ref{lem:homeomorphic subspace}, below.
\end{proof}

As a remark, a good example for the state on subnetwork $X'$ is eventually continuous and contains no defect is to let it be all-0 state at a moment, which is observed most of the time in simulations.

\begin{lemma}
\label{lem:homeomorphic subspace}
Under the conditions of Theorem \ref{thm:limit case theorem}, there exists a subspace $S\in P_c$, such that $p$ induces a homeomorphism from $S$ to $\Domain'$.
\end{lemma}
\begin{proof}
 First, reduce to a subnetwork $X''$ of $X'$ such that within $X''$ the convex hulls of neighbors is still enough to cover $\Domain'$, but any two distinct convex hulls intersect in at most one node. We then construct $S$ inductively from the empty set as follows:
\begin{enumerate}
  \item Select an integer time $t$ which is no earlier than $t_0$ big enough, such that every node in $X''$ has already been periodic for a long enough time. Pick a node $x\in X''(t)$ and add $(x,t)$ to $S$.
  \item For any neighbor of $x$ in $X''$, say $y$, there exists a continuous path lying in $P_c$, between $(x,t)$ and $(y,t_y)$, where $y\in X''(t_y)$ and $|t-t_y|\leq 1$ ($t_y$ is an integer time), which is mapped homeomorphically to the convex hull of $x$ and $y$ in $\Domain'$, because continuity holds on edge $[x,y]$, and $U_{x}\times[t,t+1)\cup U_{y}\times[t_y,t_y+1)$ is a path connected set. For any $x$'s neighbors $y$ that has not been visited, add $(y,t_y)$ with the continuous paths between $(x,t)$ and $(y,t_y)$ to $S$.
  \item Repeat step 2 for every newly visited node, until every node in a connected component of $X''$ has been visited.
\end{enumerate}
Such procedure could not be realized only if there is a cycle in $X''$, such that the continuous lift of the path to $P_c$ is not a loop, which means the state restricted on the loop is a defect. However there is no defect in $u_t(X'')$ when $t$ is big enough, by Lemma \ref{lem:invariant degree}. 
Therefore the procedure is well-defined. Start above procedures until every node in $X''$ has been visited.

Such an $S$ is mapped onto $\Domain'$ by $p$, because every convex hall of two neighbors, say $x$ and $y$, is mapped onto from the path between $(x,t_x)$ and $(y,t_y)$. The restriction of $p$ to $S$ is also injective, because every node and edge is only visited once, and $p$ restricted on every continuous path between $(x,t_x)$ and $(y,t_y)$ is homeomorphism.

The only thing left to be proved is that there exist a continuous inverse of $p|_{S}$. Let $f$ be a map from $\Domain'$ to $S$, such that $f$ maps every node $x$ in $X''$ to $(x,t_x)$ in $S$, and maps every edge between node $x$ and $y$ to the continuous path between $(x,t_x)$ and $(y,t_y)$. Such $f$ is an inverse of $p|_{S}$, and is continuous: for a point in $\Domain'$ that is not a node, it's covered by a convex hall of two neighboring nodes in $X''$, thus its small neighborhood maps to the lift of the the convex hall in $S$ continuously; for a node point $x$ in $X''$, its neighborhood maps to a neighborhood of the lift $(x,t_x)$ homeomorphically, by the procedure of constructing $S$. Thus $f$ is an continuous inverse of $p$ restricted on $S$, thus $p$ induces a homeomorphism from $S$ to $\Domain'$.
\begin{figure}
\centering
\epsfig{file=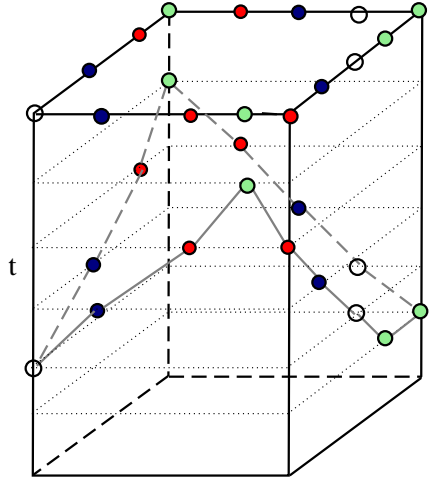}
\caption{Product space and $S$: space $\Domain'$ on top, time increases from top to bottom. State space is $\mathbb{Z}_4$, with white for state 0, light green for state 3, red for state 2, and dark blue for state 1. Gray curve represent $S$. $S\cong \Domain'$ by $p$. Continuous curves connecting top and bottom without intersecting $\bar{S}$ do not exist.}
\end{figure}
\end{proof}

\section{Main theorem and proofs}
\label{sec:proof}

\subsection{Assumptions}

The main theorem for this paper, Theorem \ref{thm:main theorem}, shows that any evader in the evasion game on a narrow hallway space $\Domain\subset\real^2$ will lose, given the appropriate assumptions about the density of the network $X$ and the initial condition. Specifically, we assume:
\begin{enumerate}
\item the projection from Rips complex $\Rips_r(X)$ to $\Domain$ preserves homotopy type;
\item each sensor $x\in X$ covers a convex set $U_x\subset\Domain$ around its location;
\item the convex hull of sensors that are pairwise neighbors is covered by the union of coverage of those sensors;
\item there is at least one seed in the initial condition.
\end{enumerate}

According to a theorem of ~\cite{Chambers_vietoris-ripscomplexes}, we will be able to build the correspondence between the Rips complex $\Rips_r(X)$ of a planar point set and its shadow $S(\Rips_r(X))$ in $\mathbb{R}^2$.

\begin{theorem}
\label{thm:iso shadow}
[ \cite{Chambers_vietoris-ripscomplexes}]
For any set of points in $\mathbb{R}^2$, $\pi_1(\Rips_r(X))\rightarrow\pi_1(S(\Rips_r(X)))$ is an isomorphism.
\end{theorem}

\begin{definition}
\label{def:local hole}
A \style{local hole} in the Rips complex is a non zero element of $\pi_1(\Rips_r(X))$ that has trivial projection in $\pi_1(\Domain)$.
\end{definition}

In the sense of local holes, Theorem \ref{thm:iso shadow} is saying that in our case, the Rips complex $\mathcal{R}_{r}(X)$ has no local hole if and only if its shadow $S(\Rips_r(X))$ has no local hole.

Another useful fact is that with very high probability, when the network is dense enough, the Rips complex $\mathcal{R}_{r}(X)$ has no local holes ~\cite{Niyogi04findingthe}. Therefore, if with enough sensors uniformly distributed in the domain and with high probability, the Rips complex $\mathcal{R}_r(X)$, and its shadow $S(\Rips_r(X))$ both have no local hole.

\subsection{Wave propogation}

\begin{definition}
\label{def:boundary node}
A \style{boundary path} along a boundary component, is defined as a simple path such that every node on the path has a coverage that intersects with the corresponding boundary, and the intersection of the coverage of every two neighbors, $x$ and $y$ on the path, also intersects the boundary nontrivially.  A boundary of a network $X$, $\partial X$ on $\Domain$ is a collection of boundary paths, one with each component of $\partial\Domain$. Refer to Figure \ref{fig:waveproof} for illustration.
\end{definition}

\begin{definition}
For positive integer set $A$, define the \style{depth $A$} nodes, $X_{A}$, as the set of all the nodes that are with depth $k\in A$ in the directed forest $\mathcal{F}$ built on the network.
\end{definition}

\begin{definition}
A connected sub network $X'$ of $X$ makes a \style{barrier}, if there exist a piece of hallway $\tilde{\Domain}$, which intersects $\partial\Domain$ at $\partial\tilde{\Domain}$, and the composition $\partial\circ i_*: H_1(\tilde{\Domain},\partial\tilde{\Domain})\rightarrow H_0(\partial\Domain)$ of $i_*: H_1(\tilde{\Domain},\partial\tilde{\Domain})\rightarrow H_1(\Domain,\partial\Domain)$ and $\partial: H_1(\Domain,\partial\Domain)\rightarrow H_0(\partial\Domain)$ is an injection, such that $X'$'s coverage contains at least one element in a nonzero class of $H_1(\tilde{\Domain},\partial\tilde{\Domain})$. In other words, it covers a region that divides the hallway locally and transversally as in Figure \ref{fig:barrier}.
\begin{figure}
\centering
\epsfig{file=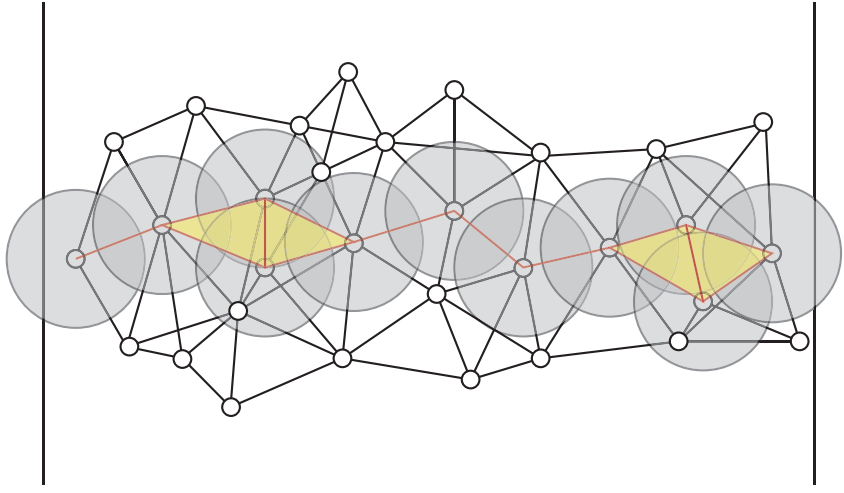,scale=0.8}
\caption{A piece of hallway with a subnetwork that is a barrier.}
\label{fig:barrier}
\end{figure}
\end{definition}

\begin{theorem}
\label{thm:wave propagate}
Let $X$ be a connected and compact network on a narrow hallway space $\Domain$, running under GHM, whose initial condition contains at least one seed; if there is a moment that the subnetwork $X'$ in a sub domain $\Domain'$ is continuous and contains no defect,
and boundary paths $\partial X'$ exists, then if at time $t$, there is a wavefront that makes a barrier, which is not supported on any end leaves of the forrest $\mathcal{F}$, then there would be a wavefront also makes a barrier at time $t+1$.
\end{theorem}

\begin{proof}
If there is a wavefront of nodes with depth $k$ at that makes a barrier, then we want to prove that a wavefront of nodes with depth $k+1$ exist which also makes a barrier. Let $A$ denote the subcomplex on subnetwork $X_{\leq k+1}$, and let $B$ denote the subcomplex on subnetwork $X_{\geq k+1}$. Then $A\cap B$ is precisely the subcomplex with nodes with depth $k+1$. On the other hand, $A\cup B$ is the whole complex, because for every simplex in the whole complex, their vertices are pairwise neighbors, so by continuity of states on $X'$ (by the assumption that there is a moment that the subnetwork $X'$ is in state zero), their depths could only differ at most by 1, by Corollary \ref{cor:continuity holds}, which means the simplex is either in $A$ or in $B$. The Mayer-Vietoris sequence for $A$ and $B$ gives:
\begin{equation}
\begin{aligned}
    H_{1}(A\cap B, \partial X')& \stackrel{\phi}{\longrightarrow}H_{1}(A,\partial X')\oplus H_{1}(B,\partial X')\\& \stackrel{\psi}
        {\longrightarrow} H_1(A\cup B, \partial X')
\end{aligned}
\end{equation}
Let $[\alpha]\in H_{1}(A,\partial X')$, $[\beta]\in H_{1}(B,\partial X')$, where $\alpha$ and $\beta$ are both connecting boundary nodes of different sides. Such an $\alpha$ exists because of the existence of a previous wavefront of depth $k$, and $\beta$ exists because the network is sufficiently dense in $\Domain$, and $\alpha$ is not supported on any end leaves of $\mathcal{F}$. Then $\psi([\alpha],[\beta])=0$ (if not, let $\beta$ be of opposite orientation) in $H_1(A\cup B, \partial X')$, because first homology of $A\cup B$ is trivial. Therefore $\psi([\alpha],[0])$ and $\psi([0],[\beta])$ are homologous. Thus $([\alpha],[\beta])\in \ker \psi$. By exactness, $\ker \psi = \mathrm{im}\ \phi$, thus there exist a $\gamma$, such that $\phi([\gamma]) = ([\alpha], [\beta])$. As $\phi$ is induced by inclusion maps, $\gamma$ has to be a path connecting boundaries of two sides, which is a wavefront of depth $k+1$, {\em cf.} Figure \ref{fig:waveproof}. Therefore, by induction, barrier-inducing wavefronts of every depth exist.
\begin{figure}
\centering
\epsfig{file=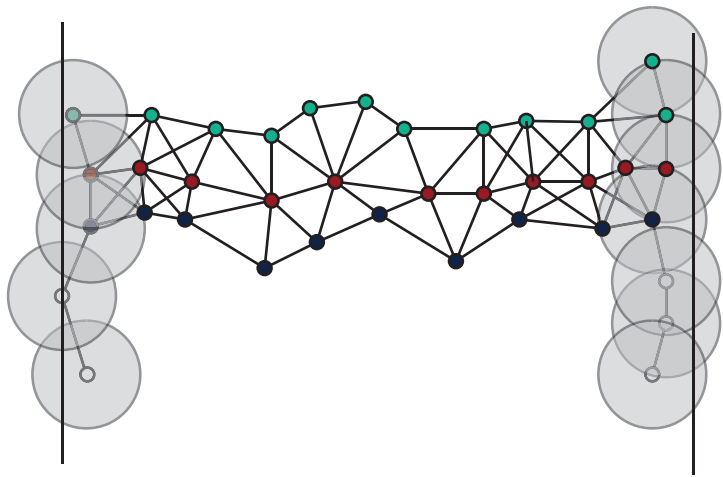}
\caption{Wavefronts propagation: light green nodes are with depth $k-1$, red nodes are with depth $k-1$, dark blue nodes are with depth $k+1$; Boundaries of the domain are covered by boundary paths.}
\label{fig:waveproof}
\end{figure}
\end{proof}

The above wave propagation theorem presents how waves travel along the hallways, but did not mention the generation of waves. The following proposition would explain how a first wave is generated under same assumptions.

\begin{proposition}
\label{prop:first wave}
With the assumptions of Theorem \ref{thm:wave propagate}, at least one wavefront that is a barrier and intersects with boundary paths on both sides must be generated.
\end{proposition}

\begin{proof}
Let $X_{\leq k}$ be the set of nodes with depth less than or equal to $k$. Then there is a filtration of Rips complexes:
\begin{displaymath}
\mathcal{R}_r(X_{\{0\}})\subset \mathcal{R}_r(X_{\leq 1}) \subset \dots \subset \mathcal{R}_r(X_{\leq k}) \subset \dots \mathcal{R}_r(X)
\end{displaymath}
Since they grow by attaching nodes within communication distance as $k$ increases, therefore, as $\mathcal{R}_r(X)$ is connected, there has to be a $k_0$, such that $\mathcal{R}_r(X_{\leq k})$ are all connected for $k\geq k_0$.
For two boundary nodes of $X_{\{k_0\}}$, if they belong to boundary paths near different boundaries, since they are connected, and for the same argument from Theorem \ref{thm:wave propagate} by using the Mayer-Vietoris sequence, they are connected by a path with nodes from $X_{\{k_0\}}$. This path generates a wavefront that is a barrier.
\end{proof}

The above results not only explain why the evader has to lose the evasion game, but also explain the behaviors of the wavefronts seen in simulations. After the first several steps, the nodes far away from the seed are all turned on, until wavefronts generated by the seeds reach them. The movements of wavefronts are verified to be moving away from seeds, and they provide locally separating barriers, as observed. Another significant property we observe from simulations is that the wavefronts make turns when reaching a corner, as shown in Figure \ref{fig:corner}. This reminds again that the behavior of the system does only depend on topology, not geometry, of the underlying space.

\begin{figure}
\centering
\epsfig{file=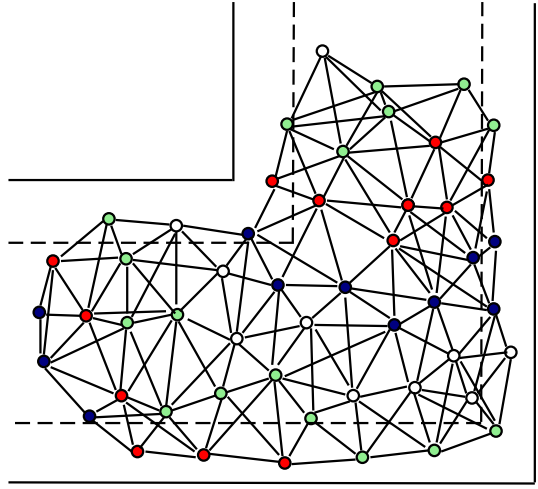}
\caption{corner of a hallway: with state space $\mathbb{Z}_4$, white for state 0, light green for state 3, red for state 2, and dark blue for state 1. The outer side boundary path have more nodes than the inner boundary path, but more nodes stay in the same states: four in light green and three in red. Thus the wavefronts propagate from vertical to horizontal.}
\label{fig:corner}
\end{figure}

\subsection{Main theorem}

For now, we will start argue that under certain conditions, evader will always lose the evasion game.

\begin{definition}
For $x\in X$, define the \style{open star} \cite{hatcher} of $x$, $\tilde{U}_x\subset \Rips_r(X)$, as the union of $x$ and all open simplices with $x$ as one vertex.
\end{definition}

\begin{lemma}
\label{lem:escape in one simplex}
Let $\sigma$ be a $d$-simplex in $\Rips_r(X)$. If there is a continuous function $f: [t_1,t_2)\rightarrow S(\sigma)$, such that
$f(t)\notin \bigcup_{x\in X(t)\cap\sigma}U_x, \ \forall t\in[t_1,t_2)$, then there exists a continuous function $\tilde{f}: [t_1,t_2)\rightarrow \sigma$, such that $\tilde{f}(t)\notin \bigcup_{x\in X(t)\cap\sigma}\tilde{U}_x, \ \forall t\in[t_1,t_2)$.
\end{lemma}

\begin{proof}
For a 2-simplex $\sigma=[x_0,x_1,x_2]$. There exist a homotopy equivalence $h$ from $S(\sigma)$ to itself, such that the interior of $\bigcap_{i=0,1,2}U_{x_i}\cap S(\sigma)$ is mapped onto the interior of $S(\sigma)$, which is $\bigcap_{i=0,1,2}\tilde{U}_{x_i}$, and the inverse image of every open edge $S([x_i,x_j])$ belongs to $U_{x_i}\cap U_{x_j}\cap S(\sigma)$. Therefore, we can construct $\tilde{f}$ as $h\circ f$, with the property that $h^{-1}(\bigcap_{i\in A}\tilde{U}_{x_i})\subset \bigcap_{i\in A}U_{x_i}$, which induces $\tilde{f}(t)\notin \bigcup_{x\in X(t)\cap\sigma}\tilde{U}_x, \ \forall t\in[t_1,t_2)$.
\end{proof}

\begin{lemma}
\label{lem:escape in complex}
For the Rips complex $\Rips_r(X)$, if there exists a continuous function $f: [t_0,\infty)\rightarrow S(\Rips_r(X))$, such that $f(t)\notin \bigcup_{x\in X(t)} U_x, \ \forall t\in[t_0,\infty)$, then there exists a continuous function $\tilde{f}: [t_0,\infty)\rightarrow \Rips_r(X)$, such that $\tilde{f}(t)\notin \bigcup_{x\in X(t)} \tilde{U}_x, \ \forall t\in[t_0,\infty)$.
\end{lemma}

\begin{proof}
The 2-complex $C_2(X)$ as a sub complex, has the same shadow as the Rips complex $\Rips_r(X)$. Lift the path $f$ from the shadow $S(C_2(X))$ to the complex $C_2(X)$, then apply lemma \ref{lem:escape in one simplex} on every simplex it goes through. This will give a lift $\tilde{f}: [t_0,\infty)\rightarrow C_2(X)\subset\Rips_r(X)$, such that $\tilde{f}(t)\notin \bigcup_{x\in X(t)} \tilde{U}_x, \ \forall t\in[t_0,\infty)$.
\end{proof}

\begin{theorem}
\label{thm:no escape in shadow}
For a network $X$ with coverage regions $\mathcal{U}$ and Rips complex $\Rips_r(X)$ with the shadow the whole $\Domain$.  Then if the state on $X$ is eventually continuous, and if there is no defect in the initial condition, then the evader loses the evasion game.
\end{theorem}
\begin{proof}
Suppose there is a continuous path $f$ for the evader to follow in order to win the evasion game, $f:[t_0,\infty)\rightarrow S(\Rips_r(X))=\Domain$, then by Lemma \ref{lem:escape in complex}, there exists a lift of $f$, $\tilde{f}: [t_0,\infty)\rightarrow \Rips_r(X)$, such that following $\tilde{f}$, the evader could win the evasion game with coverage regions $\{\tilde{U}_x|x\in X\}$. Furthermore, since $\tilde{f}(t)\notin \bigcup_{x\in X(t)} \tilde{U}_x, \ \forall t\in[0,\infty)$, and by the fact that a 1-simplex is covered by a subset of sensors that covers the simplex containing it, we can construct a continuous path $f'$ that travels only on the 1-skeleton of $\Rips_r(X)$ and still is safe, never being detected. However, by the same argument as in Theorem \ref{thm:limit case theorem}, since there is no defect in initial condition, such a strategy does not exist: any such evader would lose the game.
\end{proof}

If may not be the case that $S(\Rips_r(X))\supset\Domain$. Our approach for solving this problem is by adding sensors to the network without changing the coverage, but enlarge the Rips complex such that it projects onto the whole domain.

\begin{lemma}
\label{lem:same shadow}
If a boundary path exists within distance $\sqrt{3}/2r$ to each boundary component of $\partial\Domain$, then there is a new sensor network $\tilde{X}$ by adding sensors to $X$, with the same coverage at every moment, such that the shadow of $\Rips_r(\tilde{X})$ is $\Domain$.
\end{lemma}
\begin{proof}
For every node $x$ in the boundary path, add a node $x'$ in $U_x\cap\partial\Domain$ to the new network $\tilde{X}$, and for every edge on the path $[x,y]$, add a node $z'$ in $U_x\cap U_y\cap\partial\Domain$ to $X'$. For a quadrangle with vertices $x,y,x',y'$, it is covered by union of $U_x$ and $U_y$. Let $x'$ and $z'$ have same coverage and states as $x$, and $y'$ has the same as $y$ after equilibrium, then the coverage of $\tilde{X}$ is exactly the same as that of $X$ at every moment. Another property worth noticing is $\Rips_r(\tilde{X})$ now has its shadow same as $\Domain$, because $[x,x'], [x',z'], [x,z'], [y,z'], [z',y'], [y',y]$ are all 1-simplices in $\Rips_r(\tilde{X})$, which makes the shadow exactly $\Domain$.
\begin{figure}
\centering
\epsfig{file=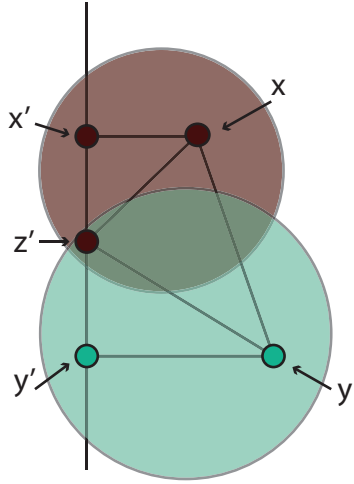}
\caption{add new nodes $x',y',z'$ to the network, with $x'$ and $z'$ have same state and coverage as $x$, and $y'$ has same state and coverage as $y$}
\end{figure}
\end{proof}

\begin{theorem}[Main Theorem]
\label{thm:main theorem}
With the existence of boundary paths within distance $\sqrt{3}/2r$ to boundary of hallways, the evader will always lose the evasion game in the sub domain $\Domain'\subset\Domain$ on which the states is eventually continuous and contains no defects.
\end{theorem}
\begin{proof}
By Theorem \ref{thm:no escape in shadow} and Lemma \ref{lem:same shadow}.
\end{proof}

\section{Controlling the Cohomology}
\label{sec:seedoff}

We have observed in simulation that sometimes there is no ``local defect'' continuously generating wavefronts, but the system still reaches a nonzero equilibrium, with the remaining wavefronts propagating along hallways in periodic way. This phenomenon contributes to the existence of a ``global defect'', which differs from the ``local defect'' in that the cycle on which the defect is supported is in a non-zero class in first homology of the Rips complex $\Rips_r(X)$, instead of a trivial one. Such an equilibrium presents a much higher portion of nodes in state $0$ than those with local defects. We will try to manually generate such patterns in GHM by turning off local defects.

Such protocol is not energy efficient unless we shift state $0$ to sleep state as follows: the new interpolation lets state $1$ to be waking state, state $2$ to be broadcasting state, and state $3$ till $0$ to be sleeping state. Then most of the nodes will be sleeping after they are eventually periodic.

Recall from Definition \ref{def:degree}, the degree (or winding number) of a continuous state on a cycle is an index measuring how man times the states cycle through the alphabet on the cycle. Therefore, after local defects are turned off by breaking the links between state $0$ and state $1$ nodes, the degree of a cycle which makes a nonzero class in first homology of the Rips complex $\Rips_r(X)$ is determined by the number of wavefronts already generated and their directions of propagation. In other words, degree for all cycles is determined absolutely by local defects' location and the number of wavefronts they have sent out in the hallways. Note that the degree is invariant in time for a continuous state. Thus counting the degree for a cycle after defects are turned off is not a difficult problem: following the direction of this cycle, the number of wavefronts in the same direction minus the number of wavefronts in the opposite direction determines the degree.

\begin{definition}
\label{def:degree of generator}
For a class $[\alpha]$ in $H_1(\Rips_r(X))$ (or $H_1(C_2(X))$, the first homology of the 2-complex), define the degree of a continuous state $u$ on $[\alpha]$ as the degree of $u$ on cycle $\alpha$. If the projection $\pi: \Rips_r(X)\rightarrow\Domain$ induces an isomorphism $\pi_*: H_1(\Rips_r(X))\rightarrow H_1(\Domain)$, then define the degree of $u$ on $\pi_*([\alpha])$ as the degree of $u$ on $\alpha$.
\end{definition}

As a remark, the degree of a generator in first homology is well-defined, if for homologous cycles $\alpha$ and $\beta$, degree of $u$ restricted on both are the same. By Lemma \ref{lem:defect discontinuity}, degree of $u$ restricted on $\alpha-\beta$, which is a null-homologous cycle, has to be zero. Thus, the degree on $\alpha$ and $\beta$ have to be the same. By abuse of notation, $\degree(u,\alpha)$ will be used for $\alpha$ as a first homology class in either $\Rips_r(X)$ or $\Domain$.

\begin{definition}
\label{def:first cohomology}
Let $\cont(X)$ represent the set of continuous states on $X$. Define a cohomologizing map $h: \cont(X)\rightarrow H^1(\Rips_r(X))=Hom(H_1(\Rips_r(X)),\zed)$, such that $h(u)([\alpha])=\degree(u,[\alpha])$.
\end{definition}

As a remark, the first cohomology $H^1(\Rips_r(X))$ defined here is a simplicial cohomology. It is torsion free and therefore can be treated as $Hom(H_1(\Rips_r(X)),\zed)$.

\begin{definition}
\label{def:wave}
A \style{single wave} is a continuous state $u$ on $X$, such that (1) there exists a barrier on which $u$ is supported, and (2) there exists a cycle $\alpha$ on which the degree of $u$ is 1.
\end{definition}

By the definition of a wave, the degree is zero on those cycles that do not intersect the wave's support. The waves move (changing supports in time) in a way that degrees are invariant. They are even additive under some circumstances, by next lemma, which allows for algebraic manipulations.

\begin{lemma}
\label{lem:addativity}
Let $\phi_1$ and $\phi_2$ be two continuous states on $X$ with supports $X_1$ and $X_2$, with no two nodes from $X_1$ and $X_2$ being neighbors. Let $\phi=\phi_1+\phi_2$ be a state on $X$, then $\phi$ is a continuous state on $X$, which satisfies $h({\phi})=h({\phi_1})+h({\phi_2})$.
\end{lemma}
\begin{proof}
The continuity of $\phi$ inside $X_1$ and $X_2$ is inherited form the continuity of $\phi_1$ and $\phi_2$. If $x_1\in X_1$ and $x_2\not\in X_1$ are neighbors, then $x_2\not\in X_2$. This means $\phi(x_2)=0$, which makes $\phi$ on the pair $(x_1,x_2)$ is continuous. The same argument works for a pair of neighbors in and out of $X_2$. For two neighbors both outside $X_1$ and $X_2$, on which $\phi$ is 0, the continuity also holds since the values have to be both 0.
Let $\alpha=\sum_{i=1}^{K}[a_i,b_i]$ be a cycle in $X$, then $\alpha\cap X_1$ and $\alpha\cap X_2$ are two non neighboring subsets, and:
\begin{displaymath}
h(\phi)=1/n\sum_{i=1}^{K}(\phi(b_i)-\phi(a_i))
\end{displaymath}
For these pairs of neighbors $a_i$ and $b_i$, there could be at least one in $X_1$, which sum up to be $h({\phi_1})$, or at least one in $X_2$, which sum up to be $h({\phi_2})$, otherwise, both are in neither $X_1$ or $X_2$, which sum up to 0. Therefore, $h({\phi})=h({\phi_1})+h({\phi_2})$.
\end{proof}

\begin{corollary}
\label{cor:sum of waves}
If states $\phi_1,\dots,\phi_k$ have distinct and non-neighboring supports, then $h(\sum\phi_i)=\sum h(\phi_i)$.
\end{corollary}

An important property of the narrow hallways $\Domain$ is has the topological type of a planar graph $G$; specifically, $G$ is a deformation retraction of $\Domain$, with retraction map $r: \Domain\rightarrow G$ and injection map $i: G\rightarrow\Domain$.  Suppose $H_1(\Rips_r(X))=H_1(\Domain)=\bigoplus_g\zed$, and $\{[\alpha_1],\dots,[\alpha_g]\}$ is a basis for $H_1(\Rips_r(X))$, accordingly, $\{\pi_*([\alpha_1]),\dots,\pi_*([\alpha_g])\}$ is a basis for $H_1(\Domain)$. Since the degree of $u$ on a cycle in $\Rips_r(X)$ is totally determined by integers $\degree(u,\alpha_1),\dots,\degree(u,\alpha_g)$ by Lemma \ref{lem:degree of summation}, we only need to focus on controlling the degree on a basis.

One problem we care about is whether one can realize every possible degree. In other words, the question could be reformed as whether the map $h$ is surjective. Specifically, is it possible to realize a continuous state $u$, such that $\degree(u,[\alpha])=f([\alpha])$, where $f: H_1(\Rips_r(X))\rightarrow \real$ is any integer valued linear map satisfying $f([\alpha+\beta])=f([\alpha])+f([\beta])$.

Our last theorem concerns this ability to {\em program pulses} in the network for customizing the response.

\begin{theorem}
\label{thm:degree realizability}
The map $h$ is surjective: if $[f]\in H^1(\Rips_r(X))$, then there exist a continues state $u$ on $X$, such that $h(u)=[f]$.
\end{theorem}
\begin{proof}
We start by selecting a specific basis for $H_1(G)$, using the standard basis of the complement of a spanning tree $T$: each  remaining edge corresponds with an element in a basis of $H_1(G)$. Let this basis be $\{[\alpha_1'],\dots,[\alpha_g']\}$, and the edges in corresponding sequence be $e_1,\dots,e_g$, where each $e_i$ is contained in only one element $\alpha_i'$.
For each $i$, there exist at least one single wave $\phi_i$ that is supported only on a subnetwork in $r^{-1}(e_i)$, and satisfies $h(\phi_i)([\alpha_j])=\delta_{ij}$, and $h(-\phi_i)([\alpha_i])=-1$. From the density assumption on the network $X$, those waves can be supported on non-neighboring subnetworks and therefore we can sum $[f]([\alpha])$ of them up to obtain a continuous state $\phi_i'$ such that $h(\phi_i')([\alpha])=[f]([\alpha])\delta_{ij}$ by Corallary \ref{cor:sum of waves}. From the same argument, $\sum\phi_i'$ is a continuous state which maps to $[f]$ under $h$.
\end{proof}

\section{Link Failure Analysis}
\label{sec:link failure}
Reliability of links is a serious issue for achieving stability of WSN \cite{Akyildiz02wirelesssensor, wsn_yick}; in practice, stability is not guaranteed, as wireless communication quality is unpredictable under different environmental and other physical conditions \cite{packet_delivery}. For our GHM system, it is important to keep communication stable, especially the links between sensors of state $0$ and state $1$, since they will determine those nodes' state at the next time step.

In this section, we will assume that every link works well with a fixed probability $\pofs$, as a more practical GHM system. By modifying our simulation accordingly, we observe that most of the nodes goes to state $0$ after the first several steps, as before. Afterwards, either the system dies out if there is no defect (either local or global defects), or wavefronts are generated around local defects. But these defects do not guarantee the system's periodicity, since link failure might result in their dying, with a probability associated with $\pofs$.

For a fixed network $X$, if given an initial state $u$ with at least one local defect, the probability that one local defect dies after $T$ time steps is a function $f$ of $X$, $u$, $T$ and $\pofs$. The smaller $\pofs$ is, the bigger the probability of defect dying. Meanwhile, $f(X,u,T,\pofs)$ is an increasing function of $T$, which approaches 1 as $T$ goes to infinity.

Although local defects die eventually almost surely, it does not affect pattern propagation. For a continuous state of waves with no local defect, which are what remain in the network after all local defects die, it could either be in a trivial cohomology class, which will die out after a while, or has at least one global defect. As in the latter case, wave propagation is not necessarily the same as in the deterministic model, since a state 0 wavefront may not turn into a state 1 wavefront. However, even this wavefront does not update to state 1 as a whole, it is of great chance that at least one of the nodes on the wavefront successfully update to state 1 (which still makes a global defect), and therefore will gradually correct the neighbors states by contact.

\section{Conclusion}
\label{sec:conclusion}

In this paper, we provide a decentralized, coordinate-free, energy-efficient intruder-detection protocol based on the Greenberg-Hastings cyclic cellular automata. The system could easily be adapted to real indoor environment if using sensing devices functioned with communication and proper sensing ranges. It displays coherence  in the sense that it is a self-assembling system with random initial conditions; its efficiency comes from low power-consuming property inherited from the scheme of the CCA. Demonstrations in \S\ref{sec:GHM} are evidence that the system behaves as intended, and this paper gives both intuition and rigor about how and why the system  works:
\begin{itemize}
\item Wave patterns are explained as a topological phenomenon, determined and described by the existence of defects with nonzero degree.
\item Assigning to wavefronts a cohomology class reveals the qualitative structure of the wavefront patterns, greatly clarifying certain classical results about CCA on lattices.
\item A non zero restriction of a cohomology class to a subdomain corresponds with a set of strategies with which the evader could win the evasion game; meanwhile, a zero restriction stands for the failure of the evader: the cohomology class is the {\em obstruction} for the pursuer to win. This pleasantly resonates with the role of cohomology in obstruction theory.
 \end{itemize}

\bibliographystyle{abbrv}
\bibliography{citation}

\end{document}